\documentclass[12pt, a4paper, reqno]{amsart}

\usepackage{a4wide}
\usepackage{amscd, amsmath,amssymb}
\usepackage[all]{xy}
\usepackage[utf8x]{inputenc} 
\usepackage{amsfonts}                      
\usepackage{enumerate}
\usepackage{color}
\usepackage{cancel}

\usepackage{hyperref}
\usepackage{nicefrac}
\usepackage{booktabs}
\usepackage{xcolor}
\usepackage{threeparttable}
\newtheorem{thm}{Theorem}
\newtheorem{lem}[thm]{Lemma}
\newtheorem{cor}[thm]{Corollary}
\newtheorem{prop}[thm]{Proposition}

\theoremstyle{definition}

\newtheorem{rem}[thm]{Remark}

\newcommand{\LieG}{{\mathsf G}}

\newcommand{\LieH}{{\mathsf H}}

\newcommand{\LieE}{{\mathsf E}}
\newcommand{\LieF}{{\mathsf F}}
\newcommand{\LieO}{{\mathsf{O}}}

\newcommand{\LieSO}{{\mathsf{SO}}}

\newcommand{\LieSp}{{\mathsf{Sp}}}
\newcommand{\LieSU}{{\mathsf{SU}}}

\newcommand{\LieU}{{\mathsf{U}}}

\newcommand{\Iso}{{\mathsf{Iso}}}

\newcommand{\Trans}{{\mathsf{Trans}}}

\newcommand{\lieG}{{\mathfrak{g}}}
\newcommand{\lieH}{{\mathfrak{h}}}

\def\N{{\mathbb N}}

\def\Z{{\mathbb Z}}
\def\R{{\mathbb R}}
\def\C{{\mathbb C}}
\def\Q{{\mathbb Q}}

\title[$\Gamma$-structures and symmetric spaces]{$\Gamma$-structures and symmetric spaces}

\author[B.\ Hanke \& P.\ Quast]{Bernhard Hanke and Peter Quast}

\address{Institut f\"{u}r Mathematik, Universit\"{a}t Augsburg, 
86135 Augsburg, Germany}
\email{hanke@math.uni-augsburg.de,\; peter.quast@math.uni-augsburg.de}



\begin{document}

\begin{abstract}
$\Gamma$-structures are  weak forms of multiplications on  closed 
oriented manifolds. As shown by \textsc{Hopf} the rational cohomology algebras of 
manifolds admitting $\Gamma$-structures are free over odd degree generators. 
We prove that  this condition is also sufficient for the existence of 
$\Gamma$-structures on manifolds which are nilpotent in the sense of homotopy 
theory. This  includes homogeneous spaces with connected isotropy groups. 

Passing to a more geometric perspective we  show that on compact oriented Riemannian symmetric spaces with 
connected isotropy groups and  
 free  rational cohomology algebras  the canonical products  
 given  by geodesic symmetries define  $\Gamma$-structures. 
This extends work of \textsc{Albers, Frauenfelder} and 
\textsc{Solomon} on  $\Gamma$-structures on Lagrangian Grassmannians.

\end{abstract}

\maketitle

\section*{Introduction} 
In his seminal papers \cite{Ho1, Ho2} on the (co)homological structure 
of Lie groups \textsc{Hopf}  introduced the notion 
of {\em $\Gamma$-manifolds}. By definition these are closed connected oriented manifolds 
$M$ together with continuous maps 
\[
   \psi : M \times M \to M 
\]
so that the mapping degrees of the two restrictions $\psi_{x}=\psi(x,-)$ and 
$\psi^{y}=\psi(-,y)$ 
are non-zero for some (and hence for all) $x, y \in M$. In some sense such maps  $\psi$, which we 
call  {\em $\Gamma$-structures}, capture the simplest non-trivial homological information of  
Lie group multiplications. 
\textsc{Hopf} proved that the rational cohomology rings of $\Gamma$-manifolds are of a surprisingly restricted 
type: Because they admit compatible  comultiplications (and are hence  {\em Hopf algebras} in modern terminology) 
they are free graded commutative $\Q$-algebras, whose generators must be in odd 
degrees as $M$ is finite dimensional, 
 see \cite{Ho1, Ho2} and \cite[Chap.\ VI, \S 2.A]{D} for further information.
As
carried out by \textsc{Borel} in \cite[Chapitre II, $\S$ 7]{Bo}, further divisibility restrictions on the mapping 
degrees of $\psi_{x}$ 
and $\psi^{y}$ have similar  implications for the cohomology rings over finite fields. 
 \par

A closely related and much better known structure is that of an  {\em H-space}, with both restrictions 
$\psi_{x}$ and $\psi^{y}$ being homotopic to the identity. 
While every compact connected oriented manifold that is an 
H-space is obviously a $\Gamma$-manifold, the converse fails: All odd dimensional 
unit spheres are
 $\Gamma$-manifolds (see \cite{Ho2}), 
 but \textsc{Adams} celebrated 
 `Hopf invariant one theorem' in \cite{Ad}
 says that only spheres of dimension $1$, $3$ or $7$ admit
 H-space structures.
  
 In the first part of our paper, Section \ref{SEC:Homotopy}, we shall make some general remarks on the 
 existence of $\Gamma$-structures on manifolds satisfying the above cohomological condition. 
 As the only non-trivial requirement 
 for a $\Gamma$-structure $\psi : M \times M \to M$ 
 is the non-vanishing of the mapping degrees of $\psi_{x}$ and $\psi^{y}$,  a construction of such structures 
 by obstruction theory requires the separation of rational and torsion information in the homotopy 
 type of $M$. This is the underlying idea of {\em rational homotopy theory}, which works 
 best for spaces whose Postnikov decompositions consist of principal fibrations and can hence 
 be described by accessible cohomological invariants. Examples are   {\em simple} spaces, 
 whose fundamental groups are abelian and  act trivially on higher homotopy groups, and, more
 generally, {\em nilpotent} spaces, whose  fundamental groups are nilpotent and 
   act nilpotently on higher homotopy 
groups. 

In Section \ref{SEC:Homotopy} we will prove the following 
converse of Hopf's result. From now on the notion {\em free algebra} stands for {\em free graded commutative 
algebra over the rationals}. 

\begin{thm} \label{existence}  Let $M$ be a closed connected oriented manifold which is 
nilpotent as a topological space. If  $H^*(M; \Q)$ is a free algebra -
necessarily over odd degree generators - then $M$ admits a $\Gamma$-structure. 
\end{thm} 

This is in sharp contrast to the existence of H-space structures, which is a much more restrictive property. 

\begin{cor} \label{homogenous} 
Let $M = \LieG/\LieH$ be a compact connected homogeneous space where  $\LieG$ is a Lie group and $\LieH < \LieG$ is a 
closed connected subgroup. 
If $H^*(M;\Q)$ is a free algebra, then $M$ admits a $\Gamma$-structure. 
\end{cor} 

Theorem \ref{existence} also implies (see Corollary \ref{free_virtab}) that nilpotent manifolds with 
free rational cohomology algebras have  virtually abelian fundamental groups.

Our abstract 
existence result motivates the search for explicit geometric constructions of $\Gamma$-structures. 
Already \textsc{Hopf} \cite{Ho2} used geodesic symmetries on odd dimensional spheres 
to write down $\Gamma$-structures. 
It is therefore natural to consider a Riemannian symmetric space $P$
endowed with its  \emph{canonical product}
(see e.g.\ \cite[Chap.\ II, \S 1]{L-I})
 \begin{equation}
 \label{EQ: Product sym space}
\Theta:P\times P\to P,\quad (x,y)\mapsto s_x(y),
 \end{equation}
 where $s_x$ denotes the \emph{geodesic symmetry} of $P$ at the point $x\in P,$ that is
the involutive isometry of $P$ that fixes $x$ and that reverses 
 the direction of all geodesics emanating from $x.$
 In Section \ref{SEC: can prod sym space} we prove:

\begin{thm}
 \label{THM: classification} 
 Let $P$ be a compact symmetric space with transvection group
 $\LieG,$ that is $\LieG$ is the connected closed subgroup of the
isometry group of $P$  generated by products of two geodesic symmetries of $P.$ 
 Assume that the isotropy subgroup 
 $\LieH < \LieG$ of a base point in $P$ within $\LieG$
 is connected. 
Then the following assertions 
are equivalent:
\begin{itemize}
\item $H^*(P;\Q)$ is a free algebra.
\item The canonical product $\Theta$ of $P$ is a $\Gamma$-structure.
\end{itemize}
\end{thm}

We use the assumption that $\LieH$ is connected in Lemma \ref{LEM: cohomologies iso}, but we do not 
know whether there is any example 
of an oriented compact symmetric space whose isotropy groups within its transvection group are not connected and whose 
rational cohomology is a free algebra. Notice however that Lemma 17 and therefore Theorem 3
still hold true, if one assumes that our compact symmetric space P can be written as a quotient of
a connected Lie group by a closed connected subgroup, or, in fact,  that it is just a nilpotent topological space.

Theorem  \ref{THM: classification}  covers Lagrangian Grassmannians of odd rank, which amounts to the main result in \cite{AFS} 
(see Remark \ref{REM: Lagrangian Grassmannian}).

We observe that whenever  the canonical product $\Theta$ defines  a 
$\Gamma$-structure on $P$, then the mapping degrees of 
$\Theta_{x}$ and $\Theta^{y}$ are (up to sign)  powers of $2$ so that, in view of  \cite{Bo}, 
our results imply that additional restrictions 
on the cohomology algebras over finite fields of characteristic different from 
two are implied  by properties of the rational cohomology algebras. This generalizes
\cite[Cor.\ 3.3 \& 4.10]{Ar2}.\par

Unfortunately Theorem \ref{THM: classification} does not cover all the manifolds from Corollary \ref{homogenous}. For 
example 
it remains an open problem to provide a geometric construction of $\Gamma$-structures on complex 
and quaternionic Stiefel manifolds (see \cite[Thm.\ 3.10, p.\ 119]{MT}), which are not symmetric spaces.

\subsection*{Acknowledgements}
We are grateful to \textsc{Jost-Hinrich Eschenburg, Urs Frau\-en\-felder, Oliver Goer\-tsches, 
Dieter Kot\-schick, Mar\-kus Up\-meier}, and \textsc{Mi\-cha\-el Wie\-meler} 
for valuable remarks 
during the preparation of this manuscript. We also thank the anonymous referee for a
helpful feedback.
The research of the first named author was supported 
by DFG grant HA 3160/6-1.

%%%%%%%%%%%%%%%%%%%%%%%%%%%%%%%%%%%%%%%%%%%%%%%
%%%%%%%%%%%%%%%%%%%%%%%%%%%%%%%%%%%%%%%%%%%%%%%
%%%%%%%%%%%%%%%%%%%%%%%%%%%%%%%%%%%%%%%%%%%%%%%

\section{Postnikov decompositions and $\Gamma$-structures} 
\label{SEC:Homotopy} 

In this section we present  a  homotopy theoretic construction of $\Gamma$-structures.
Recall  that the homotopy type of a  path connected CW complex $X$
can be analysed by means of its {\em Postnikov decomposition}, see, for example,  \cite[p. 410 ff.]{Ha1}: 
Choose a base point in $X$ and, for $n \geq 0$,  let $X_n$ be obtained by killing 
all homotopy groups of $X$ above degree $n$ by attaching cells of dimension at least $n+2$.
Up to homotopy equivalence we can assume that each inclusion $p_{n+1}: X_{n+1} \to X_n$ is a fibration. 
The long exact homotopy sequence shows that the fibre is an Eilenberg--MacLane 
space $K(\pi_{n+1}, n+1)$ where $\pi_{n+1} := \pi_{n+1}(X)$. In particular, $X_1$ is 
 the classifying space $B\pi_1(X)$. 

Recall that $X$  is called {\em simple}, if $\pi_1 = \pi_1(X)$ is 
 abelian and acts trivially on higher homotopy groups. 
For simple $X$ each fibration 
$p_{n+1}: X_{n+1} \to X_n$  is {\em principal}, see \cite[Theorem 4.69]{Ha1},  that is
the pull back of the path-loop fibration 
\[
   K(\pi_{n+1}, n+1) \to PK(\pi_{n+1} , n+2) \to K(\pi_{n+1}, n+2)
 \]
along a map $X_{n} \to K(\pi_{n+1}, n+2)$. By definition this map 
determines    the $n$-th {\em $k$-invariant}  
$k_n \in H^{n+2}(X_{n}; \pi_{n+1})$.
This class is equal to the image of the fundamental class in $H^{n+1}(K(\pi_{n+1} , n+1) ; \pi_{n+1})$ 
under the transgressive differential $d_{n+2}$ in the Leray--Serre spectral sequence for the fibration $p_{n+1}$. Furthermore, the $k$-invariant $k_n$ is equal to zero, if and only if the 
fibration $p_{n+1}: X_{n+1} \to X_n$ is fibre homotopy equivalent to the trivial fibration. 
We denote by $(k_n)_{\Q} \in H^{n+2}(X_{n}; \pi_{n+1} \otimes \Q)$ the image of $k_n$ under 
the coefficient homomorphism $\pi_{n+1} \to \pi_{n+1} \otimes \Q$. 

In the following we collect some well known facts on the cohomology of Eilenberg--MacLane 
spaces. 

\begin{lem} \label{computation} 
Let $C$ be a (finite or infinite) cyclic group, and let $n >0$ be a positive integer. Then
\begin{itemize} 
     \item $H^*(K(C,n) ; \Z)$ is a finitely generated group in each degree. 
    \item for $C = \Z$ the cohomology algebra $H^*(K(\Z,n) ; \Q)$ is free with one generator in degree $n$. 
   \item for $|C| < \infty$ the reduced cohomology $\widetilde{H}^*(K(C,n); \Q)$ is equal to $0$.
  \end{itemize} 
Let $m > 0$ be a positive integer and let 
$\mu_m : C \to C$ be multiplication by $m$. Then
\begin{itemize}
   \item for $C = \Z$  the induced map 
\[
      \mu_m^* : H^*(K(\Z,n) ; \Q) \to H^*(K(\Z,n) ; \Q) 
\]
is given by multiplication with $m^k$ on $H^{kn}(K(\Z,n) ; \Q)$ for $k \geq 0$. 
\item for all $C$ the map 
\[
   \mu_m^* : \widetilde{ H}^* (K(C,n) ; \Z/m ) \to \widetilde{H}^*(K(C,n) ; \Z/m)
\]
is equal to $0$. 
\end{itemize} 
\end{lem}  

\begin{proof} We first prove all but the last statement by induction on $n$.
For $n=1$ we have $K(C,1) = BC$, the classifying space of $C$, so that the assertions are clear 
for $C = \Z$ (recall  $B\Z = S^1$).
For $|C| < \infty$ the classifying space $BC$ is an infinite dimensional lens space, the cohomology 
$\widetilde{H}^*(BC; \Z)$ is equal to $\Z / |C|$ in even degrees and $0$ in odd degrees, and $\mu^*_m : 
H^*(BC ; \Z) \to H^*(BC; \Z)$ is 
given by multiplication with $m^k$ on $H^{2k}(BC ; \Z)$. Together with the universal coefficient 
theorem this  completes the case $n =1$.

For the induction step we recall  that  the cohomology with coefficients in some commutative ring 
$R$ of base and fibre of the path 
loop fibration 
\[
     K(C, n) \to PK(C, n+1) \to K(C,n+1) 
\]
appear on the two coordinate axes of the $E_2$-term of the Leray--Serre spectral sequence and 
that this spectral sequence is natural with respect to homomorphisms $C \to C$. 
The spectral sequence converges to $H^*(PK(C, n+1); R)$, which vanishes in positive degrees, because $PK(C,n+1)$ is contractible. 

By induction this shows that $H^*(K(C, n+1); \Z)$ is 
finitely generated in each degree, that $\widetilde{H}^*(K(C, n+1); \Q) = 0$ for $|C| < \infty$, and  that $H^*(K(\Z,n+1); \Q)$ is a free algebra in one generator of degree $n+1$.  The last implication is based 
on the multiplicative structure of the spectral sequence. 

Now let $m > 0$ and $\mu_m : \Z \to \Z$ be multiplication by $m$. Then the naturality of the spectral sequence shows inductively 
that $\mu_m^* : H^{*}(K(\Z, n+1); \Q) \to H^*(K(\Z, n+1); \Q)$ is multiplication 
by $m^k$ in degree $k(n+1)$. This finishes the induction step. 

It remains to show the last statement of Lemma \ref{computation}: For all cyclic groups $C$ and all $m, n > 0$ the map 
\[
   \mu_m^* : \widetilde{ H}^* (K(C,n) ; \Z/m ) \to \widetilde{H}^*(K(C,n) ; \Z/m)
\]
is equal to $0$.  

First, let us assume 
that $m = p^r$, where $p$ is a prime number and $r > 0$.  If $C$ is finite, then
the inclusion of the (unique) Sylow $p$-subgroup $P \to C$ induces an isomorphism 
\[
   \widetilde{ H}^* (K(C,n) ; \Z/p^r ) \cong \widetilde{H}^*(K(P, n); \Z/p^r) 
\]
by a spectral sequence argument. It is therefore enough to concentrate on the  case $C= \Z/p^{\ell} $, ${\ell}  > 0$, if 
$C$ is finite, and $C = \Z$, if $C$ is infinite.  

We work by induction 
on $r$. Let $r =1 $, hence $m = p$.  It is well known, 
see \cite{Ca, Se}, or \cite[Theorem 6.19]{McCl}, 
that $H^*(K(C ,n); \Z/p)$ is a polynomial algebra with free  generators of the form 
$\mathcal{P}(\iota_n)$ where $\iota_n \in H^n(K(C,n);\Z/p)$ is the fundamental class  and 
$\mathcal{P}$ is some mod $ p$ cohomology operation. Because the map $\mu_p^*$ is 
multiplication by $p$ and hence zero 
on $H^n(K(C,n); \Z/p)$ the assertion for $r= 1$ is 
implied by the naturality of the operations $\mathcal{P}$. 

Next, assuming the assertion for $r-1$, the assertion for $r$ follows by use of the exact Bockstein sequence 
\[
    \cdots \rightarrow \widetilde{H}^*(K(C,n); \Z/p^{r-1}) \rightarrow \widetilde{H}^*(K(C,n); \Z/p^{r}) \to 
    \widetilde{H}^*(K(C,n); \Z/p) \rightarrow \cdots , 
\]
which is associated to the short exact coefficient sequence 
\[ 
    0 \to \Z/p^{r-1} \to \Z/ p^{r} \to \Z/p \to 0 
\]
and which is natural with respect to homomorphisms $C \to C$.  A simple diagram chase together with the 
equality  $\mu_{p^{r}} = \mu_{p^{r-1}} \circ 
\mu_p$ then shows the assertion for $r$.

After having finished the proof 
for $m = p^r$ we will now deal with the general case 
$m = p_1^{r_1} \cdot \ldots \cdot p_k^{r_k}$ with pairwise different primes  $p_i$ and $r_i > 0$. According to the decomposition 
\[
    \Z/m \cong \Z/p_1^{r_1} \times  \cdots \times \Z/p_k^{r_k} 
\]
we obtain a splitting 
\[ 
     \widetilde{H}^*(K(C, n); \Z/m) \cong \widetilde{H}^*(K(C,n); \Z/p_1^{r_1}) \oplus \cdots \oplus 
     \widetilde{H}^*(K(C,n) ; \Z/p_k^{r_k}) 
 \]
 which is natural in $C$. On each summand $ \widetilde{H}^*(K(C,n); \Z/p_i^{r_i})$  the multiplication  $\mu_m: C \to C$ induces the zero map, because $\mu_m$ 
factors through $\mu_{p_i^{r_i}}$.  This implies the 
last assertion of Lemma \ref{computation} for general $m>0$.

 \end{proof} 

Let $X$ be  a connected finite simple CW complex whose rational 
cohomology  algebra is free. It must be finitely generated with all generators in odd degrees, because 
$X$ is assumed to be finite. Serre's finiteness theorem (or a direct
inspection of the Postnikov decomposition in connection with Lemma \ref{computation}) 
implies that the homotopy groups $\pi_*(X)$ 
are finitely generated in each degree. Hence they are finite products of cyclic groups. 

\begin{lem} \label{torsion} For each $n  \geq 0 $ the following holds. 
\begin{itemize} 
    \item $H^*(X_n; \Q)$ is a free algebra with generators in degrees $\leq n$ corresponding
     to  the duals of the generators of $\pi_{i}(X) \otimes \Q$, $i \leq n$.
    \item The canonical map $X \to X_n$ induces an injective map in rational cohomology. 
    \item The rationalized $k$-invariant $(k_n)_{\Q}  \in H^{n+2}(X_{n} ; \pi_{n+1} \otimes \Q)$ vanishes. 
\end{itemize} 
 \end{lem} 

 \begin{proof} The first assertion implies the second one, because 
 the canonical map $X \to X_n$ induces an isomorphism in rational cohomology up to degree $n$
 and the cohomology algebra $H^*(X; \Q)$ is free. The second assertion implies the 
third one by the following argument. The rationalized $k$-invariant $(k_n)_{\Q}$
 is the image 
of the fundamental class 
 in $H^{n+1}(K(\pi_{n+1} , n+1) ; \Q)$ under the differential $d_{n+2}$ in the 
spectral sequence for the fibration $X_{n+1} \to X_n$. If this differential 
were non-zero, then  the induced map $H^*(X_{n}; \Q) \to H^*(X_{n+1};\Q)$ 
would not be injective. However then the map $H^*(X_n; \Q) \to H^*(X;\Z)$ would 
not be injective, either, by use  of the factorization $X \to X_{n+1} \to X_n$. This 
contradicts the second assertion. 

 It is hence enough to prove the first assertion by induction on $n$. This assertion is 
 clear for $n=0$, because $X_0$ is homotopy equivalent to a point. 
In the inductive step 
the assumption $(k_n)_{\Q} = 0$ implies 
 \[
    H^*(X_{n+1} ; \Q) \cong H^*(X_{n}; \Q ) \otimes H^*(K(\pi_{n+1}, n+1); \Q)
\]
by Lemma \ref{computation},  the K\"unneth theorem 
applied to the splitting of $\pi_{n+1}$ into cyclic groups, and the multiplicative properties of the Leray--Serre 
spectral sequence. From this the first assertion follows for $n+1$. 
\end{proof} 

\begin{prop} \label{multi} For all  $n  \geq  0$  and $m > 0$ there is a self map $f _{m,n}: X_n \to X_n$ 
with the following properties. 
\begin{itemize} 
  \item The induced map $f_{m,n}^* : H^*(X_n; \Q) \to H^*(X_n;\Q)$ is an isomorphism. 
  \item The induced map $f_{m,n}^* : \widetilde{H}^*(X_n ; \Z/m) \to \widetilde{H}^*(X_n; \Z/m)$ is equal to $0$. 
 \end{itemize} 
 \end{prop}
 
 \begin{proof} We apply induction on $n$.  Again the case $n = 0$ is clear. 
 
For the induction step we assume the assertion holds for a fixed $n$ and all  $m > 0$. 
 By Lemma \ref{torsion} 
the $k$-invariant $k_n \in H^{n+2}(X_n ; \pi_{n+1})$ is a 
torsion class.  Let $\tau$ denote the 
order of the torsion subgroup of $H^{n+2}( X_{n}; \Z)$.  Then 
 the restriction of the canonical map 
\[
     H^{n+2}(X_{n}; \Z) \to H^{n+2}(X_{n}; \Z) \otimes \Z/\tau \to H^{n+2}(X_{n} ; \Z/ \tau) 
\]
 to the torsion subgroup of $H^{n+2}( X_{n}; \Z)$ is injective. 
This is clear for the first map, and for the second 
map it follows from the universal coefficient theorem.  
In particular, the self map $f_{\tau, n} : X_n \to X_n$ provided by the 
induction hypothesis induces the zero map 
on the torsion subgroup of $H^{n+2}( X_{n}; \Z)$. 
We now consider a splitting 
\[
    \pi_{n+1} \cong  C_1 \times \cdots \times C_k
\]
into cyclic groups and obtain a  corresponding splitting 
\[
    H^{n+2}(X_n ; \pi_{n+1})  \cong H^{n+2}(X_n; C_1) \oplus \cdots \oplus H^{n+2}(X_n; C_k) 
\]
which is  natural in $X_n$. 
%With respect to this splitting the 
%torsion subgroup of $H^{n+2}(X_n ; \pi_{n+1})$ maps to the direct sum 
%of the torsion subgroups of the summands $H^{n+2}(X_n; C_i)$,  $1 \leq i \leq k$. 
Now, for each $ 1 \leq i \leq k$ there is a self map $f_i : X_n \to X_n$ inducing an isomorphism in rational cohomology 
and the zero map on the torsion 
subgroup of $H^{n+2}(X_n; C_i)$:  If  $C_i = \Z$ we  take $f_{\tau, n} : X_{n} \to X_n$
as explained before, and if 
$C_i = \Z/m$  for some $m$ we take the self map $f_{m,n} : X_n \to X_n$ provided 
by the induction hypothesis. 
Let $f := f_k \circ \cdots \circ f_1 : X_n \to X_n$. By construction we have: 
\begin{itemize} 
\item The map  $f$ induces an isomorphism $H^*(X_n; \Q) \to H^*(X_n; \Q)$. 
\item $f^*(k_n) = 0$.
\end{itemize} 

In the  pull back square of fibrations 
\[
    \xymatrix{  K(\pi_{n+1}, n+1) \ar[d]^{\rm incl.} \ar[r]^{=}& K(\pi_{n+1}, n+1) \ar[d] \\
      f^*(X_{n+1}) \ar[r]^{F}   \ar[d]  &         X_{n+1} \ar[d]^{p_{n+1}} \\ 
                                   X_n      \ar[r]^{f }    &         X_n          }
\]
 the map $F$ induces an isomorphism in rational cohomology, by 
a spectral sequence argument, and because  $f$ induces an isomorphism in rational 
cohomology. Because $f^*(k_n) = 0$  the induced fibration $f^*(X_{n+1}) \to X_n$ is fibre 
homotopy trivial and hence we get a homotopy equivalence 
$X_{n} \times K(\pi_{n+1}, n+1) \simeq f^*(X_{n+1})$. Let
\[
     \alpha_n : X_{n} \times K(\pi_{n+1}, n+1) \simeq f^*(X_{n+1}) \stackrel{F}{\to} X_{n+1} 
\]
be the composition of the resulting maps. The map $\alpha_n$ 
induces an isomorphism in rational cohomology by 
construction. 

Next, let  $\mu := \mu_{|k_n|}  : \pi_{n+1} \to \pi_{n+1}$ be the multiplication by  the order 
of the torsion class $k_n \in H^{n+2}(X_{n} ; \pi_{n+1})$. We wish do define a map $\beta_n$ 
fitting into a commutative diagram 
\[ 
  \xymatrix{   K(\pi_{n+1} , n+1) \ar[r]^{K(\mu, n+1)} \ar[d] & K(\pi_{n+1}, n+1) \ar[d]^{\rm incl.} \\
    X_{n+1}                           \ar[r]^-{\beta_n} \ar[d]^{p_{n+1}} &   X_{n} \times K(\pi_{n+1}, n+1)\ar[d]^{\rm proj.} \\ 
                                          X_n                                                     \ar[r]^{=}    &         X_n           }
 \]
 The only non-trivial task is the construction of the map $X_{n+1} \to K(\pi_{n+1}, n+1)$ appearing in the middle 
 horizontal line. For this we consider the diagram 
\[ 
  \xymatrix{   K(\pi_{n+1} , n+1) \ar[r]^{=} \ar[d] & K(\pi_{n+1}, n+1)  \ar[rr]^-{K(\mu , n+1)} \ar[d] & & K(\pi_{n+1}, n+1) \ar[d] \\
    X_{n+1}                           \ar[r]  \ar[d]  &   PK(\pi_{n+1} , n+2) \ar[rr]^-{PK(\mu  , n+1)} \ar[d] & & PK(\pi_{n+1}, n+2) \ar[d] \\
        X_n                              \ar[r]^-{k_n}    &       K(\pi_{n+1} , n+2)   \ar[rr]^-{K(\mu  , n+2)} &   &    K(\pi_{n+1} , n+2)     }
 \]
By assumption the composition in the lower row is homotopic to a constant map. We use 
the homotopy lifting property to homotop  the composition
$X_{n+1} \to PK(\pi_{n+1}, n+2)$ in the second row  to a map $X_{n+1} \to PK(\pi_{n+1}, n+1)$ which 
factors through the fibre inclusion in the right hand column.

The map $\beta_n$ induces an isomorphism in rational cohomology, again by a spetral 
sequence argument combined with Lemma \ref{computation}.

With the maps $\alpha_n$ and $\beta_n$ in hand we are ready to conclude the induction step.  Let $m > 0$ be arbitrary. 
Using the induction hypothesis, Lemma \ref{computation} and the K\"unneth formula 
it is easy to construct a self map 
\[
 f'_{m,n}:   X_{n} \times K(\pi_{n+1}, n+1) \to X_{n} \times K(\pi_{n+1}, n+1)
\]
which induces an isomorphism in rational cohomology and the zero map in reduced $\Z/m$-cohomology. 
But then the composition 
\[
  f_{m, n+1} : X_{n+1} \stackrel{\beta_n}{\longrightarrow} X_{n} \times K(\pi_{n+1}, n+1) \stackrel{f'_{m,n}}{\longrightarrow} X_{n} \times K(\pi_{n+1}, n+1) 
  \stackrel{\alpha_n}{\longrightarrow}  X_{n+1}
\]
is as required. 
\end{proof}

For further use we isolate the following information from the proof of Proposition \ref{multi}. 

\begin{cor} \label{decomp} For each $n$ there are
maps $\alpha_n: X_{n} \times K(\pi_{n+1}, n+1) \to X_{n+1}$ 
and $\beta_{n} : X_{n+1} \to X_{n} \times K(\pi_{n+1}, n+1)$ inducing
isomorphisms in rational cohomology. 
\end{cor} 

Theorem \ref{existence} for simple spaces now follows from the next proposition. 

\begin{prop} \label{exist_gamma} Let $X$ be a connected 
finite CW-complex which is a simple topological space and whose 
rational cohomology algebra  is free. Then there is 
a continuous map 
 \[
   \psi : X \times X \to X 
\]
 so that for all $x,y  \in X$ the maps $\psi(x, -): X \to X$ and $\psi(-,y): X \to X $ induce isomorphisms in rational cohomology. 
\end{prop} 

\begin{proof} It is clear that such a map 
exists on $X_0$. Assume that we have already constructed a map $\psi_n : X_n \times X_n \to X_n$ 
with the required properties. 
 
Together with the product on $K(\pi_{n+1} , n+1)$
induced by the addition map on $\pi_{n+1}$ we obtain a multiplication $\psi_{n+1}'$ 
on $X_{n+1}' := X_n \times K(\pi_{n+1},n+1)$ 
with the required properties. Now we define $\psi_{n+1}$ as the composition 
\[  
  X_{n+1} \times X_{n+1} \stackrel{\beta_{n} \times \beta_{n} }{\longrightarrow}  X_{n+1}' \times 
   X_{n+1}' \stackrel{\psi_{n+1}'}{\longrightarrow} X'_{n+1} \stackrel{\alpha_{n}}{\longrightarrow} X_{n+1} .
\]
where $\alpha_{n}$ and $\beta_{n}$ are taken from Corollary \ref{decomp}.

Once we have constructed $\psi_{n}$ with $n > 2 \dim X$, the construction of $\psi$ is complete by 
the cellular approximation theorem. 
\end{proof} 

For the next notions and results compare \cite[Sections 3.1. and 3.2]{MP}. 
A path connected CW complex  $X$ is called {\em nilpotent}, if $\pi_1(M)$ is 
a nilpotent group and acts nilpotently on the higher homotopy groups. 
Note that simple complexes are automatically nilpotent. 
For a nilpotent complex the fibrations $p_{n+1}: X_{n+1} \to X_n$ in the Postnikov decomposition 
are in general not principal, but they admit 
{\em finite principal refinements}
\[
      X_{n+1} =: Y_{r_n} \stackrel{q_{r_n}}{\longrightarrow} Y_{r_n - 1} \stackrel{q_{r_n-1}}{\longrightarrow} \ldots Y_1 \stackrel{q_1}{\longrightarrow} Y_0 := X_n,
 \]
 where each $q_i : Y_i \to Y_{i-1}$ is a principal fibration with fibre $K(G_i, n+1)$ for some abelian group $G_i$, which 
 is classified by a cohomology class in $H^{n+2}(Y_{i-1} ; G_i)$. The same argument as in the proof 
 of Proposition \ref{exist_gamma} then shows:

 \begin{prop} \label{exist_nilpotent} Let $X$ be a connected finite CW complex which is a nilpotent topological space. Assume that 
 the rational cohomology algebra of $X$ is free. Then there is a map 
 $\psi: X \times X \to X$ so that  for all $x,y  \in X$ the maps $\psi(x,-) : X \to X$ and $\psi(-,y): X \to X$ 
 induce isomorphisms in rational cohomology. 
 \end{prop} 

This proposition immediately implies Theorem \ref{existence}. 

\begin{lem} \label{simple_space}  Let $M = \LieG / \LieH$ be a homogeneous space where $\LieG$ is a connected Lie group 
and $\LieH < \LieG$ 
be a closed connected subgroup. Then $M$ is a simple topological space.
\end{lem} 

\begin{proof} This fact is well known and we include a proof for the readers' convenience. 
Because $\LieH$ is connected we get an exact sequence
\[
    \ldots \to \pi_2(\LieG/\LieH) \to \pi_1(\LieH) \to \pi_1(\LieG) \to \pi_1(\LieG/\LieH) \to 1.
\]
The topological group $\LieG$ has abelian fundamental group, and  hence the same holds for $\LieG/\LieH$. 

Next, let $f : S^n \to \LieG/\LieH$ and $\mu : S^1 \to \LieG/\LieH$ be based maps, where we take the south pole of 
any sphere as base point.
Recall that $[\mu]_*([f]) \in \pi_n(\LieG/\LieH)$, the result of the action of $[\mu] \in \pi_1(\LieG/\LieH)$
on $[f]\in \pi_n(\LieG/\LieH)$, 
 is represented by the following map $\mu *  f : S^n \to \LieG/\LieH$. Consider the one point union 
$S^n \vee S^n$ where the second sphere is piled above the first one, identifying the north pole of the 
first with the south pole of the second. We take the south pole of the first sphere as base point of $S^n \vee S^n$. 
Now consider the composition $S^n \stackrel{\text{height}}{\longrightarrow}  [0,1] \stackrel{\mu}{\rightarrow} \LieG/\LieH$ 
on the lower sphere, and the map $f : S^n \to \LieG/\LieH$ on the upper sphere, and compose 
the resulting map $S^n \vee S^n \to \LieG/\LieH$ 
with the base point preserving coproduct $S^n \to S^n \vee S^n$. This defines $\mu * f$. 
Note that for $n=1$ this results in the usual conjugation action of $\pi_1(\LieG/\LieH)$ on itself. 

By  the above exact sequence the map $\mu$ lifts to a based
map $\overline{\mu} : S^1 \to \LieG$. Using the left multiplication  of $\LieG$ on $\LieG/\LieH$ and 
the above explicit description of $\mu * f$ 
it is easy to show that $\mu * f$ is based homotopic to $f$. \end{proof}

Together with our previous results this implies Corollary \ref{homogenous}.  

%The structure results of rational graded Hopf algebras also cover algebras with generators in arbitrary large 
%degrees: By  \cite{Ho2} they are isomorphic to tensor products of a free algebra with generators in even 
%degrees and an exterior algebra with generators in odd degrees. Note that in the cohomology rings 
%of (finite dimensional) manifolds no even degree generator can occur. But if we allow 
%infinite dimensional spaces Proposition \label{exist_nilpotent} can be generalized as follows. 

%\begin{prop} Let $X$ be a path connected nilpotent CW complex with finite skeleta. Assume that $H^*(X; \Q)$ 
%is a free algebra (hence a product of a polynomial algebra and an exterior algebra). Then there is a product map 
%\[  
%     \psi : X \times X \to X 
%\]
%so that for each $x \in X$ the induced maps on $\psi(x, -)$ and $\psi(-,x)$ are injective in rational homology. 
%\end{prop} 

%The proof is similar to the proof of Proposition \ref{exist_nilpotent} above, with the only exception that 
%the induction cannot be finished after finitely many steps in the Postnikov decomposition. 

% is in accordance with the fact \cite{Ho2} that rational graded Hopf algebras which are finitely generated in 
%each degree are free graded algebras with finitely many generators in each degree. Because this 
%general fact is not important for the main results of our paper we omit a detailed proof at this point. 

\begin{rem} Of course the above argument is modeled along the lines of 
rational homotopy theory. In particular our Proposition \ref{multi} is implied
by \cite[Theorem (12.2)]{Sul} (note that by Lemma \ref{torsion} our $X$ is a formal space 
in the sense of rational homotopy theory). However we found it somewhat difficult to trace 
a complete  proof of this theorem in the literature. 
%However, a complete proof of Sullivan's theorem 
%has been carried out in the literature only for simply connected spaces \cite{Sh} by 
%obstruction theory based on cellular decompositions. Because finite skeleta of 
%non-simply connected nilpotent spaces are in general not nilpotent \cite{Le}, 
%the non-simply connected case requires a different approach based on Postnikov decompositions
%(as indicated already in \cite{Sul}). 
We feel that  this fact and the special focus of our paper justifies 
the  ad hoc, but self-contained discussion above  instead of an in-depth exploration 
of rational homotopy theory. 
\end{rem} 

We are grateful to \textsc{Dieter Kotschick} for pointing out the following lemma and corollary.

\begin{lem}\label{LEM: virtually abelian} Each $\Gamma$-manifold has virtually abelian fundamental group. 
\end{lem} 

\begin{proof} Let $\psi: M \times M \to M$ be a $\Gamma$-structure and  $x, y \in M$. Without loss of generality
we can assume  that $\psi_x, \psi^y  : M \to M$ preserve a base point in $M$. 
Let $H_1, H_2 < \pi_1(M)$ be defined as the images of 
\[
      (\psi_x)_* : \pi_1(M) \to \pi_1(M) \; ,\quad  (\psi^y)_* : \pi_1(M) \to \pi_1(M) . 
\]
The map  $\psi_{x}, \psi^y : M \to M$ factor through the connected coverings of $M$ 
defined by $H_1$ and $H_2$, respectively. Because $\psi_x$ and $\psi^y$ have non-zero 
mapping degrees, these coverings are finite and hence $H_1$ and $H_2$ are of finite index in $\pi_1(M)$. 
This implies that also $H_1 \cap H_2 < \pi_1(M)$ is of finite index. 

The map $\psi_* : \pi_1(M) \times \pi_1(M) \to \pi_1(M)$ being a group homomorphisms, elements 
in $H_1$ commute with elements in $H_2$. This implies that the finite index 
subgroup $H_1 \cap H_2 < \pi_1(M)$ is 
abelian. 
\end{proof} 

Together with Theorem \ref{existence} this implies

\begin{cor} \label{free_virtab} Let $M$ be a closed connected oriented manifold which is 
nilpotent as a topological space. If  $H^*(M; \Q)$ is a free algebra -
necessarily over odd degree generators - then $\pi_1(M)$ is virtually abelian. 
\end{cor} 

It remains an interesting open problem 
whether this conclusion can be drawn without the use of Theorem \ref{existence}.

\section{Canonical products on  symmetric spaces}
\label{SEC: can prod sym space}
In this section we prove Theorem \ref{THM: classification}. One implication is a special case 
of Hopf's theorem \cite{Ho1, Ho2}. Let $P$ be 
a compact symmetric space whose isotropy groups within its transvection group are connected. 
This assumption implies that $P$ is orientable. 
 We are left to show that if the rational cohomology  
 $H^*(P;\Q)$ of $P$ is a free algebra, then the
 canonical product  $\Theta$  on  $P$ defined in Equation 
 \eqref{EQ: Product sym space} is a $\Gamma$-structure.\par 

We first observe that the degree of the map
 $\Theta_x:P\to P,\; y\mapsto \Theta(x,y)=s_x(y)$
with $x\in P$ fixed is $$\det(\Theta_x)=(-1)^{\dim(P)}.$$ \par
For fixed $y\in P$ we will examine  the map
 \begin{equation}
\label{EQ: theta}
 \theta:=\Theta^y:P\to P,\quad x\mapsto  \Theta(x,y)=s_x(y).
\end{equation}
If the degree of $\theta$ does not vanish, then
$\Theta$ is a $\Gamma$-structure. 
We will reduce our considerations to irreducible simply connected symmetric spaces of compact type. 
As arguments we use some features of compact symmetric spaces, which 
can be found in the classical literature such as  \cite{He,L-I, L-II}
or \cite[Chap.\ 8]{W}.\par 

We start with three preliminary lemmata.
Just like compact Lie groups (see e.\ g.\ \cite[Thm.\ 4.29, p.\ 198]{Knapp}) compact symmetric spaces 
admit finite covers that split off flat factors:

\begin{lem}
\label{LEM: split covering}
 Every compact symmetric space $P$ is finitely covered by a product 
 $T\times \widetilde{Q}$ of a flat torus $T$ and a simply connected compact symmetric space
 $\widetilde{Q}.$ 
\end{lem}

\begin{proof}
The deck transformation group $\Delta$ of the universal cover $\widetilde{P}$ of $P$  is a finitely
generated
discrete subgroup of the abelian centralizer $C_{\Iso(\widetilde{P})}(\Trans(\widetilde{P}))$
of the transvection group $\Trans(\widetilde{P})$ of $\widetilde{P}$ in its 
isometry group (see \cite[Lem.\ 1.2, p.\ 194]{Sakai} 
and \cite[Thm.\ 8.3.11, p.\ 244]{W}).
Since $\widetilde{P}\cong\R^k\times \widetilde{Q},$ where $\widetilde{Q}$ is a simply connected compact 
symmetric space, 
the isometry group of $\widetilde{P}$ 
splits as $\Iso(\widetilde{P})=\Iso(\R^k)\times \Iso(\widetilde{Q}).$\par
Thus any element of $\Delta$ has the form  $f\times g$ for 
some $f\in C_{\Iso(\R^k)}(\Trans(\R^k))\cong\R^k$ and
some $g\in C_{\Iso(\widetilde{Q})}(\Trans(\widetilde{Q})).$ 
Since $\widetilde{Q}$ is a symmetric space of compact type 
$C_{\Iso(\widetilde{Q})}(\Trans(\widetilde{Q}))$ is finite. Let $N$ be 
a common multiple of the orders of elements of $C_{\Iso(\widetilde{Q})}(\Trans(\widetilde{Q})),$
then $g^N=e$ for all $g\in C_{\Iso(\widetilde{Q})}(\Trans(\widetilde{Q})).$
Since $\Delta$ is abelian, the $N$-th power is an endomorphism of $\Delta$ whose image $\Delta^N$
acts  trivially on the $\widetilde{Q}$ factor of $\widetilde{P}.$ 
As $\Delta^N$ has finite index in $\Delta,$ the space 
$\widetilde{P}/\Delta^N\cong T\times\widetilde{Q},$ which is the desired cover of $P,$ is compact. 
\end{proof}

\begin{lem}
\label{LEM: Riemannian products}
Let $P_1$ and $P_2$ be two compact oriented symmetric spaces.
 Then the canonical product on $P_1\times P_2$ is a $\Gamma$-structure 
 if and only if the canonical products on $P_1$ and on $P_2$ are both
$\Gamma$-structures. 
\end{lem}

\begin{proof}
The Riemannian product $P_1\times P_2$ is again a compact  oriented symmetric space and the geodesic 
symmetries of $P_1\times P_2$ are
products of geodesic symmetries of $P_1$ and of 
$P_2.$ From this the claim follows easily by the multiplicativity of 
mapping degrees. 
\end{proof}

\begin{lem}
\label{LEM: Coverings}
Let $p:\widehat{P}\to P$ be an orientation preserving  Riemannian covering
between two compact oriented symmetric spaces.
Then the canonical product on $\widehat{P}$ is a $\Gamma$-structure 
 if and only if the canonical product on $P$ is a $\Gamma$-structure.
\end{lem}

\begin{proof}
The canonical products $\widehat{\Theta}$ on $\widehat{P}$ and 
$\Theta$ on $P$ are related by
$$\Theta\circ(p\times p)=p\circ \widehat{\Theta}.$$
Let $\widehat{y}\in\widehat{P}$ be a chosen origin and  $y:=p(\widehat{y}).$
Then the maps 
$\widehat{\theta}=\widehat{\Theta}^{\widehat{y}}:\widehat{x}\mapsto \widehat{s}_{\widehat{x}}(\widehat{y})$
and $\theta=\Theta^y: x\mapsto s_x(y)$ satisfy
$p\circ \widehat{\theta}=\theta\circ p.$
Since $p$ is a covering, we get 
$$\deg(p)\deg(\widehat{\theta})=\deg(\theta)\deg(p),$$
where $\deg(p)$ coincides with  the number of sheets of $p.$
Division by $\deg(p)$ yields
$\deg(\widehat{\theta})=\deg(\theta).$
\end{proof}

Using these lemmata, we  can proceed with our proof of Theorem \ref{THM: classification}. 
Let $P$ be a symmetric space as in
Theorem \ref{THM: classification} such that 
$H^*(P;\Q)$ is an exterior algebra generated by homogeneous elements in
odd degrees. By Lemma \ref{LEM: split covering} and the de Rham decomposition for symmetric spaces 
there is a Riemannian product 
$$\widehat{P}:=T\times \widetilde{P}_1\times \dots \times \widetilde{P}_m$$
of a flat torus $T$  (points and circles are  considered zero and one dimensional 
tori) 
and  irreducible simply connected compact symmetric 
spaces $\widetilde{P}_1,\dots, \widetilde{P}_m$, $m \geq 0$,  together with 
a finite Riemannian covering map $p : \widehat{P} \to P$.

\begin{lem}
\label{LEM: cohomologies iso} 
The covering $p:\widehat{P}\to P$ induces an isomorphism between the 
 graded $\Q$-algebras  $H^*(\widehat{P}; \Q)$  and
$H^*(P;\Q).$ 
\end{lem} 

\begin{proof} 
We prove this claim in two different ways. 
First, let $\widehat{y}\in \widehat{P}$ and $y:=p(\widehat{y}).$ 
Let $\widehat{\LieG}$ and $\LieG$ denote the transvection groups 
of $\widehat{P}$ and $P.$ The isotropy group
$\widehat{\LieH}\subset \widehat{\LieG}$ 
of $\widehat{y}$ is connected, because $\widehat{P}$  is a product of a torus and simply connected 
compact symmetric spaces, and the isotropy group $\LieH\subset \LieG$ of $y$
is connected by assumption.
Their linear isotropy actions are identified by $p.$ 
By \cite[Thm.\ 8.5.8]{W} the real cohomology algebra
of any compact symmetric space $S$ is 
isomorphic to the algebra of those elements in $\bigwedge^*T_xS,\; 
x\in S,$ which are invariant under the action of the isotropy subgroup 
of $x$ within the transvection group of $S.$
Thus $p$ induces an isomorphism $H^*(\widehat{P};\R)\cong H^*(P;\R)$ 
and also an isomorphism of the rational cohomologies.\par

Second, the fundamental group $\pi_1(\widehat{P})  < \pi_1(P)$ is abelian and acts trivially on 
$\pi_n(\widehat{P}) = \pi_n(P)$ for $n > 1$, because $P = \LieG / \LieH$ with connected $\LieH$ and
by Lemma \ref{simple_space}. Hence $\widehat{P}$ is a simple space. Furthermore, because $p$ 
is a finite covering, the induced map $\pi_{*}(\widehat{P}) \otimes \Q \to \pi_*(P)\otimes \Q$ is an isomorphism.
 By the Whitehead--Serre theorem for simple spaces (or inspecting the induced map between Postnikov 
decompositions of $\widehat{P}$ and $P$) the induced map in rational cohomology is an isomorphism as well. 
\end{proof} 

Lemma \ref{LEM: cohomologies iso} and K\"unneth's formula imply
$$H^*(P;\Q)\cong H^*(T;\Q)\otimes H^*(\widetilde{P}_1;\Q)\otimes\dots 
\otimes H^*(\widetilde{P}_m; \Q).$$
Note that  $H^*(P;\Q)$ is  generated by homogeneous elements in odd degrees if and only if 
the same holds for all 
 $H^*(\widetilde{P}_j;\Q),\; j\in\{1,\dots, m\}.$ 
Since the mapping degree
of $\theta$ on a $r$-dimensional flat torus is 
$2^r,$  we are left to verify Theorem \ref{THM: classification} only
for irreducible simply connected 
compact symmetric spaces $P=\LieG/\LieH$ whose rational cohomology algebra is generated by homogeneous 
elements in odd degrees.\par

Since the  rational cohomology of an \emph{inner} compact symmetric space, this is a compact symmetric 
space all of whose 
geodesic symmetries belong to its transvection group $\LieG,$ 
 has only  contributions in even degrees
  (see e.g.\ \cite[proof of Thm.\ 8.6.7]{W} and \cite[Thm.\ VII, p.\ 467]{GHV}), 
  we may assume that $P=\LieG/\LieH$  is an \emph{outer} symmetric space.
  Using the classification of simply connected irreducible compact symmetric spaces one could at this point determine all 
  outer symmetric spaces that satisfy the assumptions of 
  Theorem \ref{THM: classification} by 
checking  case-by-case the rational cohomology of such spaces given in \cite{Take} (see also \cite{MT, Spiv, GHV}). But we
prefer a  more conceptional approach that we essentially learned from \textsc{Oliver Goertsches.}\par
Since the Lie algebras of $\LieG$ and of $\LieH$ form a \emph{Cartan pair}  $(\lieG,\lieH)$
 (see \cite[pp.\ 448 \& 465]{GHV}) with $\rho:=\mathrm{rank}(\lieG)-\mathrm{rank}(\lieH)>0,$ 
 $H^*(\LieG/\LieH;\Q)$ is isomorphic to a tensor product of a  $2^{\rho}$ dimensional exterior algebra
  and a quotient of a symmetric algebra (see \cite[Thm.\ IV, p.\ 463]{GHV} and \cite[Thm.\ 3]{Kotschick}). 
Therefore
 $\LieG/\LieH$ satisfies the hypotheses of Theorem \ref{THM: classification} if and only if 
 $\dim(H^*(\LieG/\LieH;\Q))
 =2^{\rho}.$ 
 \begin{rem}
 In \cite{Goertsches}  \textsc{Goertsches} gave a Lie theoretic description of these spaces. 
 They are precisely those symmetric spaces where the number of Weyl chambers of $\lieG$ 
that intersect a given Weyl chamber of $\lieH$ is equal to one. 
Non-trivial intersections of Weyl chambers of $\lieG$ with Weyl chambers 
of $\lieH$ are called \emph{compartments} in \cite{EMQ}. \par
From \cite[Sect.\ 3]{Murak}  
one sees that these spaces $P$ are those where 
 the involution of $\lieG$ associated with $P$ 
is the canonical extension 
of an order two automorphism of the Dynkin diagram of $\lieG$ 
(see also \cite[pp.\ 33ff.]{Bur-Raw} and 
\cite[pp.\ 1128 \& 1129]{EMQ}).
 \end{rem}

The simply connected irreducible compact outer  symmetric spaces $P=\LieG/\LieH$ of this kind are precisely 
(see \cite[Table p.\ 305]{Murak}):
\begin{itemize}
 \item those of splitting rank. These are symmetric spaces where the rank of $\LieG$ is 
 the sum of the rank of $\LieH$ and the rank of $P. $ The irreducible simply connected compact symmetric 
 spaces of splitting rank are the
simply connected compact simple Lie groups, the odd dimensional round spheres,
$\LieSU_{2n}/\LieSp_n$ with $n\geq 3,$ and the exceptional space 
$\LieE_6/\LieF_4.$
 \item  and the spaces $\LieSU_{2n+1}/\LieSO_{2n+1}$ with $n\in\N.$ 
 \end{itemize}
The referee made us aware 
that the mapping degree of $\theta$  has already been
calculated  by \textsc{Araki} in these cases:
\begin{itemize}
 \item $\deg(\theta)=2^{\mathrm{rank}(P)},$ if $P$ is of splitting rank, see \cite[Thm.\ 3.1]{Ar2}, and
 \item  $\deg(\theta)=2^n$ for $P=\LieSU_{2n+1}/\LieSO_{2n+1},$ see \cite[Thm.\ 4.9]{Ar2}.
\end{itemize}
This concludes the proof of Theorem \ref{THM: classification}.

\begin{rem} 
 In \cite{Ar2} \textsc{Araki} actually considered the map
 $\LieG/\LieH\to\LieG/\LieH,\; g\LieH\mapsto g\sigma(g^{-1})\LieH,$ where
 $\sigma$ is the involution of $\LieG$ such that $\LieH$ is the identity component of 
its fixed point set. In our terms we have
 $\sigma(g)=s_{e\LieH}\circ g\circ s_{e\LieH}.$
Using $s_{e\LieH}(g\LieH)=\sigma(g)\LieH$ and $s_{g\LieH}=g\circ s_{e\LieH}\circ g^{-1}$
one sees that Araki's map coincides with our map $\theta$ if one chooses $eH$ as base point.
\end{rem}

\begin{rem}(Lagrangian Grassmanians of odd rank)
\label{REM: Lagrangian Grassmannian}
The Lagrangian Grassmannian $\mathcal{L}:=\LieU_{2n+1}/\LieO_{2n+1}$ can be identified with the set of all 
Lagrangian subspaces of 
$\C^{2n+1}.$  Since $\mathcal{L}=(\LieU_{2n+1}/\{\pm I\})/\LieSO_{2n+1},$
it meets the assumptions of Theorem \ref{THM: classification}. 
 The reflection at a Lagrangian subspace is an orthogonal 
 anti-symplectic involution of $\C^{2n+1}$ and 
vice-versa. Identifying $\mathcal{L}$ with the space of all orthogonal anti-symplectic involutions 
of $\C^{2n+1}$ it is  
shown in \cite{AFS} that the conjugation 
$$\mathcal{L}\times \mathcal{L}\to \mathcal{L},\quad (R_1,R_2)\mapsto R_1R_2R_1$$
 is
 a $\Gamma$-structure on $\mathcal{L}.$
 If one identifies $\mathcal{L}$ with the fixed point set $\mathrm{Fix}(\tau)$ of 
 the transposition map $\tau$ of $\LieU_{2n+1}$ by mapping the Lagrangian subspace 
 $A(\R^{2n+1}),\; A\in\LieU_{2n+1},$ to the matrix $AA^\top\in \mathrm{Fix}(\tau),$ 
 the $\Gamma$-structure defined above 
can be written as
 $$\mathrm{Fix}(\tau)\times \mathrm{Fix}(\tau)\to \mathrm{Fix}(\tau), \quad (A,B)\mapsto AB^{-1}A$$
 (see  \cite[p.\ 930]{AFS}).
 Since $\mathrm{Fix}(\tau)$ is a totally geodesic submanifold of $\LieU_{2n+1},$
 the $\Gamma$-structure considered in  \cite{AFS} is just
the  canonical product on $\mathcal{L}$ induced by the geodesic symmetries of the Lie group 
 $\LieU_{2n+1}.$
Since $\mathcal{L}$ is finitely covered by $S^1\times (\LieSU_{2n+1}/\LieSO_{2n+1})$ 
the  result in \cite{AFS} can be recovered from Lemmata \ref{LEM: Riemannian products} and \ref{LEM: Coverings}
and \cite[Thm.\ 4.9]{Ar2}.
\end{rem}

\begin{rem}[Compact Lie groups]
  A compact Lie group with a bi-invariant metric is a symmetric space of splitting rank. 
  The canonical product on a compact Lie group,
  considered as a symmetric space, discussed here 
  is a $\Gamma$-structure  different from the Lie theoretic product. If one chooses the identity as 
  base point, then
  $\theta$ is actually the squaring map.
\end{rem}

%%%%%%%%%%%%%%%%%%%%%%%%%%%%%%%%%%%%%%%%%%%%%%%
%%%%%%%%%%%%%%%%%%%%%%%%%%%%%%%%%%%%%%%%%%%%%%%

\end{document}